\documentclass[12pt]{amsart}
\usepackage{amssymb}

\numberwithin{equation}{section}

\newtheorem{theorem}{Theorem}[section]
\newtheorem{proposition}[theorem]{Proposition}
\newtheorem{lemma}[theorem]{Lemma}

\newtheorem{conjecture}[theorem]{Conjecture}

\theoremstyle{definition}

\theoremstyle{remark}
\newtheorem{remark}[theorem]{Remark}

\newcommand{\ShortExactSeq}[3]%
{0\to{#1}\to{#2}\to{#3}\to 0}
\newcommand{\DShortExactSeq}[3]%
{\xymatrix@1{0\ar[r]&{{#1}}\ar[r]&{{#2}}\ar[r]&{{#3}}\ar[r]&0}}

\newcommand{\R}{\mbox{$\Bbb R$}}

 \renewcommand{\phi}{\varphi}
\renewcommand{\epsilon}{\varepsilon}

\newcounter{ritmctr}
{\end{itemize}}
\newcounter{aitmctr}
{\end{itemize}}

\begin{document}

\title{A remark on the real coarse Baum-Connes conjecture}

\author{John Roe}
\address{Department of Mathematics, Penn State University, University
Park PA 16802}
\date{\today}
\maketitle

\section{Introduction}
Let $X$ be a complete Riemannian manifold.  The \emph{coarse assembly map} \cite{higson_baum-connes_1995, higson_analytic_2000} is a homomorphism from the locally finite complex $K$-homology of $X$ to the complex $K$-theory of the Roe algebra $C^*(X)$.  To emphasize the words ``complex'' and ``locally finite'', which are usually implicit, we write this as
\[ A\colon KU^{lf}_*(X) \to KU_*(C^*(X)). \]
 This homomorphism sends the homology class of a generalized Dirac operator $D$ to the coarse index of $D$.  In particular, if $X$ is spin and has uniformly positive scalar curvature, then  the fundamental class of $X$ (which is simply the homology class of the spinor Dirac operator) maps to zero under the coarse assembly map.

In the literature, the coarse assembly map is usually defined in \emph{complex} $K$-theory.  It is ``well known'' to experts that the statements above also hold in real K-theory, but this fact does not seem to be stated explicitly. The purpose of this note is to document the construction of the real coarse assembly map and associated versions of  the Baum-Connes conjecture and the vanishing theorem.  This result is needed in A.~Dranishnikov's approach to Gromov's conjecture relating positive scalar curvature to macroscopic dimension.~\cite{dranishnikov_macroscopic_2013}

\section{The assembly map}

Let $X$ be a complete Riemannian manifold and let $H=L^2(X)$ be the \emph{real} Hilbert space of square-summable real-valued functions on $X$.  We may define the real $C^*$-algebras $C^*(X)$ and $D^*(X)$ in the usual way: $D^*(X)$ is the real $C^*$-algebra of operators on $H$ generated by the finite propagation, pseudolocal operators, and $C^*(X)$ is the ideal generated by the finite propagation, locally compact operators.

\begin{remark} Notice that the complexifications of these real versions of $D^*(X)$ and $C^*(X)$ are the usual complex versions. \end{remark}

\begin{lemma} There is a natural isomorphism
\[ KO_{i+1}(D^*(X)/C^*(X)) \to KO_i^{lf}(X). \]
\end{lemma}

To prove this consider the real $C^*$-algebras ${\frak C}(X)$ and ${\frak D}(X)$ of \emph{all} locally compact and pseudolocal operators on $X$, respectively.  There is a natural forgetful map
\[ D^*(X)/C^*(X) \to {\frak D}(X)/{\frak C}(X) . \]
The sheaf-theoretic arguments of~\cite{roe_sheaf_2012} (see section 3.1 in that paper) carry through without change to show that this forgetful map is an \emph{isomorphism} of $C^*$-algebras.  On the other hand, the results of~\cite{roe_paschke_2004} show that there is a Paschke duality isomorphism
\[ KO_{i+1}({\frak D}(X)/{\frak C}(X)) \cong KO_i^{lf}(X). \]
Combining these we get the lemma.

Now consider the short exact sequence of real $C^*$-algebras
\[ 0 \to C^*(X) \to D^*(X) \to D^*(X)/C^*(X) \to 0. \]
Using the lemma, the boundary map in real $K$-theory coming from this short exact sequence gives a homomorphism
\[ KO_i^{lf}(X) \to KO_i(C^*(X)). \]
This is the coarse assembly map in the real case.

\begin{remark} By Voiculescu's theorem~\cite{voiculescu_non-commutative_1976}, we may replace $L^2(X)$ here by $L^2(X;V)$, where $V$ is any orthogonal vector bundle over $X$, and obtain the same assembly map. In addition, the bundle $V$ may be graded (see~\cite{roe_paschke_2004}). \end{remark}

\section{The vanishing theorem}
Let us recall how the homology class of the Dirac operator is defined.  Let $X$ be a Riemannian, spin $n$-manifold. Let $C$ denote the Clifford algebra of $\R^{n,0}$.  Let $S$ denote the canonical $C$-Dirac bundle over $X$ \cite[Section II.7]{lawson_spin_1990}, and let $D$ be its Dirac operator.  We consider the graded Hilbert space $H=L^2(S)$.

The complex version of the following result can be found in many places (for example \cite{higson_analytic_2000, roe_positive_2012}).  The real version is proved in essentially the same way; there is one change which we will explain in a moment.

\begin{lemma} In the above situation, let $f$ be a real-valued continuous function on $\R$ that has limits at $\pm\infty$. Then $f(D)\in D^*(X)^C$, and if $f$ is even (or odd) as a function on $|R$, then $f(D)$ is even (or odd) as an operator on $H$.  Moreover if $f$ tends to 0 at $\pm\infty$ then $f(D)\in C^*(X)^C$.  (The superscript $C$ denotes the part of the algebras that commutes with the right action of $C$ on $H$; in other words, $f(D)$ is $C$-linear.) \end{lemma}

Recall that the proof goes via the Fourier decomposition
\[ f(D) = \frac{1}{2\pi} \int \hat{f}(t) e^{itD} dt. \]
The only change needed in the real case is to recall that $\hat{f}(-t) = \bar{\hat{f}}(t)$ for a real-valued $f$, and so to write the integral on the right as
\[ \int u(t) \cos (tD) dt - \int v(t) \sin(tD) dt, \]
where $u$ and $v$ are the real and imaginary parts of $\hat{f}$.  This expression is valid in the real Hilbert space $H$, rather than in its complexification.

It follows that if $\chi$ is a normalizing function, the equivalence class of $\chi(D)$ in the graded algebra $D^*(X)^C/C^*(X)^C$ is a supersymmetry, defining an element of $KO_1$ of this algebra.  But, as observed in~\cite{roe_paschke_2004}, there is an isomorphism
\[ KO_1(D^*(X)^C/C^*(X)^C) \cong KO_{n+1}(C^*(X)/C^*(X)) \]
and this defines the class $[D]\in KO_{n+1}(C^*(X)/C^*(X)) \cong KO^{lf}_n(X)$.

\begin{theorem} If $X$ has uniformly positive scalar curvature then the coarse index, $A[D]$, vanishes in $KO_n(C^*(X))$. \end{theorem}

\begin{proof} From the above discussion and the long exact sequence in $KO$-theory, it suffices to show that the class of $\chi(D)$ in $KO_1(D^*(X)^C/C^*(X)^C)$ comes from some class in $KO_1(D^*(X)^C)$.  But, using the scalar curvature condition, we see that $D$ has a spectral gap near zero.  We may therefore choose a normalizing function $\chi$ which is equal to $\pm 1$ on the spectrum of $D$.  The operator $F=\chi(D)$ is then a supersymmetry in $D^*(x)^C$ defining the desired class. \end{proof}

\section{The coarse Baum-Connes conjecture}

One version of the coarse Baum-Connes conjecture is the following

\begin{conjecture}\label{cbc1} Suppose that the proper metric space $X$ is \emph{uniformly contractible} \cite{roe_coarse_1993}.  Then the coarse assembly map $A$, from the $K$-homology of $X$ to the $K$-theory of $C^*(X)$, is an isomorphism. \end{conjecture}

Notice that this conjecture has both a real version (for $KO$) and a complex version (for $KU$).

\begin{proposition} The complex version of Conjecture~\ref{cbc1} for a given space $X$ implies the real version for the same space. \end{proposition}

\begin{proof} This follows from Karoubi's descent lemma~\cite{karoubi_descent_2001} applied to the algebra $D^*(X)$.  Compare \cite{baum_baum-connes_2004}. \end{proof}

A different version of the coarse Baum-Connes conjecture does not involve the uniform contractibility hypothesis.  This version needs the notion of \emph{coarse homology}: the (real or complex) coarse $K$-homology of a proper metric space $X$ is the direct limit of the coarse $K$-homology groups of $X_n$, where $X_n$ are the nerves of uniform covers forming an anti-Cech sequence for $X$.  The spaces $X_n$ are all coarsely equivalent to one another, and to $X$, and thus the coarse assembly maps for $X_n$ pass to a limit to give an assembly map
\[ KX_*(X) \to K_*(C^*(X)), \]
where once again the letter $K$ may refer to real or complex $K$-theory.  Once again, see~\cite{roe_coarse_1993,higson_analytic_2000} for more details on this.

A more general version of the coarse Baum-Connes conjecture would then claim that the coarse assembly map
\[ KX_*(X) \to K_*(C^*(X)) \]
is an isomorphism for every proper metric space $X$.  Once again, the complex version of this conjecture implies the real version.  To see this, form direct sequences of $C^*$-algebras
\[ C^*(X_1)\to C^*(X_2)\to\ldots, \quad D^*(X_1)\to D^*(X_2)\to\ldots \]
and let $C$ and $D$ denote their direct limits.  There is an exact sequence
\[ 0 \to C \to D \to D/C\to 0, \]
since the direct limit functor is exact.  Because $K$-theory commutes with direct limits, the more general coarse Baum-Connes conjecture is equivalent to the statement that $D$ has zero $K$-theory.   By applying Karoubi's descent lemma~\cite{karoubi_descent_2001} to this algebra, we see that the complex version of this coarse Baum-Connes conjecture implies the corresponding real version.

\begin{remark} It should be noted that neither version of the coarse Baum-Connes conjecture holds for all spaces~\cite{higson_counterexamples_2002}.   What is asserted is that a \emph{particular} space $X$ will satisfy the real version if it satisfies the complex version. \end{remark}

\end{document}